\documentclass[runningheads]{llncs}
\usepackage[T1]{fontenc}
\usepackage{adjustbox}
\usepackage{wrapfig}
\usepackage{graphicx}
\usepackage{enumitem} 

\usepackage{url}
\usepackage{cmll}
\usepackage{comment}
\usepackage{amssymb}
\usepackage{color}
\usepackage{xspace}
\usepackage{bussproofs}
\EnableBpAbbreviations
\usepackage{tikz}
\usepackage{tikz-cd}
\usepackage{hyperref}

\usepackage{authblk}
\usepackage{breakcites}
\usepackage{hyphenat}
\usepackage{xcolor}

\usepackage{array}
\newcolumntype{C}[1]{>{\centering\arraybackslash}p{#1}}
\newcolumntype{L}[1]{>{\arraybackslash}p{#1}}
\usepackage{amsmath}
\usepackage{tikz}

\usepackage{setspace}

\AtBeginDocument{%
  \addtolength\abovedisplayskip{-0.5\baselineskip}%
  \addtolength\belowdisplayskip{-0.5\baselineskip}%
  \addtolength\abovedisplayshortskip{-0.5\baselineskip}%
  \addtolength\belowdisplayshortskip{-0.5\baselineskip}%
}

\newcommand{\ydnote}[1]{{\small\color{red}YD: #1}}

\newcommand{\fakeparagraph}[1]{

\textit{#1} \ \ }

\newcommand{\jty}{J^{\infty}}
\newcommand{\mty}{M^{\infty}}

\newcommand{\nomi}{\mathbf{i}}
\newcommand{\nomj}{\mathbf{j}}

\newcommand{\cnomm}{\mathbf{m}}

\newcommand{\val}[1]{[\![{#1}]\!]}
\newcommand{\descr}[1]{(\![{#1}]\!)}
\renewcommand{\phi}{\varphi}

\begin{document}
\title{Game semantics for non-distributive modal $\mu$-calculus\thanks{This project has received funding from the European Union’s Horizon 2020 research and innovation programme under the Marie Skłodowska-Curie grant agreement No 101007627. The research of Krishna Manoorkar is supported by the NWO grant KIVI.2019.001. The research of Apostolos Tzimoulis is partially supported by the Key Project of the Chinese Ministry of Education (22JJD720021). Yiwen Ding and Ruoding Wang are supported by the China Scholarship Council.
}}
%
%
\author{Yiwen Ding\inst{1}\orcidID{0009-0000-6447-6479} \and
Krishna Manoorkar\inst{1}\orcidID{0000-0003-3664-7757} \and
Mattia Panettiere\inst{1}\orcidID{0000-0002-9218-5449} \and
Apostolos Tzimoulis\inst{1,2}\orcidID{0000-0002-6228-4198} \and
Ruoding Wang\inst{1,3}\orcidID{0009-0005-3995-3225}}
\authorrunning{Y.\ Ding, K.\ Manoorkar, M.\ Panettiere, A.\ Tzimoulis, R.\ Wang}
%
\institute{Vrije Universiteit Amsterdam, The Netherlands \and Institute of Logic and Cognition, Sun Yat-Sen University, China \and
 Xiamen University, China }
\maketitle              
\begin{abstract}
In this paper, we generalize modal $\mu$-calculus to the non-distributive  (lattice-based) modal  $\mu$-calculus  and formalize some scenarios regarding categorization using it. We also  provide a game semantics for the developed logic. The proof of adequacy of this game semantics proceeds by generalizing the  unfolding games on the power-set algebras to the arbitrary lattices  and showing that  these games can be used to determine the least and the greatest fixed points of a monotone operator on a lattice.  Finally, we define a notion of bisimulations on the polarities and show invariance of non-distributive  modal  $\mu$-calculus under them. 

\end{abstract}
\section{Introduction} \label{sec:intro}
Fixed point operators and their applications have been studied extensively in mathematics \cite{granas2003fixed,agarwal2001fixed}. One of the most prominent formal languages to study fixed point operators on power-set algebras in logic  is modal $\mu$-calculus \cite{bradfield200712,venema2008lectures,kozen1983results}. Modal $\mu$-calculus provides a straight-forward  expansion  of modal logic with the least and the greatest fixed  point operators, which has several applications in formal modelling and theoretical computer science \cite{lenzi2005modal,venema2008lectures}. 

Non-distributive modal logics have been studied in recent years as a formal model for categorization theory~\cite{conradie2017toward,conradie2021rough,conradie2019logic} and evidential reasoning~\cite{conradie2020non}. The Knaster-Tarski theorem~\cite{knaster1928theoreme,tarski1955lattice} implies that any monotone operator on a complete lattice must have a least and a greatest fixed point.  Fixed point operators on complete lattices arise naturally  in categorization theory in several contexts. Many categories in social and linguistic contexts often arise as fixed points of some dynamic process or interaction.  Moreover, many classification or clustering algorithms used to categorize data  perform some operation iteratively until the outcome stabilizes, i.e.~reaches a fixed point. Another example of a fixed point operator in categorization theory is  provided by the common knowledge operator described in~\cite{conradie2017toward} which describes the epistemic view of a category commonly held by all the agents.

In this paper, we study fixed points of  monotone operators on complete lattices formally, by generalizing the modal $\mu$-calculus to the non-distributive setting. In particular, we generalize  two important  notions in classical modal $\mu$-calculus to the non-distributive setting:  game semantics \cite{NIWINSKI199699,HELLA2022104882} (a.k.a.~game-theoretic semantics) and bisimulations. More specifically,  we define an adequate evaluation game  for non-distributive modal $\mu$-calculus and show its invariance under the  notion of bisimulation on polarity-based models defined in this paper.

\fakeparagraph{Structure of the paper.} In Section~\ref{sec:prelim} we gather the preliminaries on non-distributive logic, game semantics, fixed points and classical modal $\mu$-calculus. In Section~\ref{sec:Game semantics for non-distributive modal logic} we introduce the game semantics for the non-distributive modal logic. In Section~\ref{sec:Non_distributive_modal_mu_calculus} we introduce the non-distributive modal $\mu$-calculus, and give examples of some scenarios formalized using it. In Section~\ref{sec:Game semantics for non-distributive modal mu-calculus} we introduce the unfolding games on the complete lattices, give the evaluation game for the non-distributive modal $\mu$-calculus, and show their adequacy. In Section~\ref{sec:bisimulations_polarity} we generalize the bisimulations to polarity-based models and show the non-distributive modal $\mu$-calculus is bisimulation invariant. Finally, in Section \ref{sec:Conclusion} we gather the concluding remarks and give some directions for future research.

\section{Preliminaries} \label{sec:prelim}
In this section we gather the background information on fixed points, board games, non-distributive modal logic, and the classical modal $\mu$-calculus. We take inspiration in our presentation from~\cite{venema2008lectures}. 

\subsection{Fixed points on lattices}\label{sssec:prelimfixed}
For any complete lattice $\mathbb{L}$ and monotone map $f:\mathbb{L}\to\mathbb{L}$, any $u\in\mathbb{L}$ is  a \emph{postfixed point} of $f$ if $u\leq f(u)$,  a \emph{prefixed point} of $f$ if $f(u)\leq u$,  a \emph{fixed point} of $f$ if $f(u)=u$.
The following theorem is well known:
\begin{theorem}[Knaster-Tarski]
    For any complete lattice $\mathbb{L}$, and monotone map $f:\mathbb{L}\to\mathbb{L}$, the set of fixed points of $f$ forms a complete lattice.
\end{theorem}



\noindent Thus, any monotone map $f$ on a complete lattice has a least and a greatest fixed point (indeed, the least prefixed and greatest postfixed point) denoted by $\mathrm{lfp}(f)$ and $\mathrm{gfp}(f)$, respectively. The least (resp.\ greatest) fixed point above (resp.\ below) a set $S$ of fixed points can be defined by recursion on ordinals $\alpha$:
\begin{itemize}[noitemsep,nolistsep]
    \item $f^0(u)=u=f_0(u)$, $f^{\alpha+1}(u)=f(f^\alpha(u))$ and $f_{\alpha+1}(u)=f(f_\alpha(u))$,
    \item $f^{\alpha}(u)=\bigvee_{\beta<\alpha} f^{\beta}(u)$ and $f_{\alpha}(u)=\bigwedge_{\beta<\alpha} f_{\beta}(u)$, where $\alpha$ is a limit ordinal.
\end{itemize}
Then, for any set of fixed points $S$, $(f^{\alpha}(\bigvee S))_{\alpha\in Ord}$ is an increasing sequence, and since $\mathbb{L}$ is set-size it reaches a fixed point which is the least fixed point above $S$. Likewise, $(f_{\alpha}(\bigwedge S))_{\alpha\in Ord}$ is decreasing; hence, it eventually reaches the greatest fixed point below $S$.

\subsection{Board games} \label{ssec:prelimgames}
In this paper, the  games we discuss  are played by two players, Abelard (denoted $\forall$) and Eloisa (denoted $\exists$). These games will be 
 formally described as a tuple $(B_\exists,B_\forall, \mathcal{M})$, where $B_\exists$ and $B_\forall$ are disjoint sets, that correspond to the positions on the board (denoted $B=B_\exists\cup B_\forall$) on which $\exists$ and $\forall$ are expected to move, while $\mathcal{M}:B \to\mathcal{P}(B)$ is a function that given a position indicates the admissible moves the player can make. In all the games we consider, how the game proceeds from a position  is independent of the previous play leading to the position i.e.~the admissible moves for any player  in a position 
 are completely determined by the position. 
 
 A (potentially countable) sequence $(b_n)_{n\in\omega}$ such that $b_n\in B$ and $b_{n+1}\in\mathcal{M}(b_n)$ is called a \emph{play starting from $b_0$}. If the play is finite with last element $b_k$ such that $\mathcal{M}(b_k)\neq\varnothing$, then we call the play \emph{partial}. The board is accompanied by the winning rules that describe which player wins. In the games if a play ends at a position $b\in B_\ast$ ($\ast\in\{\forall,\exists\}$) such that $\mathcal{M}(b)=\varnothing$, i.e. player $\ast$ has no move, then we say that the play is losing for $\ast$. The only other winning conditions we entertain have to do with the infinite plays: in this case, we designate $W_\exists\subseteq B^{\omega}$ the set of winning plays for $\exists$. In these games there are no ties and hence $W_\forall=B^{\omega}\setminus W_\exists$.
The winning conditions for infinite plays in such games  concern only the tails of the sequences, hence if two sequences $(b_n)_{n\in\omega}$ and $(b'_n)_{n\in\omega}$ are eventually the same, then they are winning for the same player. As both the possible moves for a game in a position depends only on the current position, and the winning conditions only depend on the tails of  sequence of positions in a play, in any position, whether $\exists$ or $\forall$ wins in that position is independent of how the position is reached. 

A {\em strategy} for a player $\ast$ for a game is a function $\sigma_\ast:B^{<\omega}\times B_\ast\to B$ such that if $(b_i)_{i\leq n}$ is a partial play, then $\sigma_\ast((b_i)_{i\leq n})\in\mathcal{M}(b_n)$, i.e.~,\ $\sigma_\ast$ takes the partial plays that end with a position that $\ast$ plays and provides a move for $\ast$. A strategy $\sigma$ is called {\em winning} for $\ast$ starting from $b_0$, if every play $(b_n)_n$ starting with $b_0$ in which $\ast$ follows $\sigma$ (that is, if $b_n\in B_\ast$, then $b_{n+1}=\sigma((b_i)_{i\leq n})$) is winning for $\ast$. A position $b_0$ is called winning for $\ast$ if $\ast$ has a winning strategy starting from $b_0$. Strategies can sometimes be simplified as functions $\sigma_\ast:B_\ast\to B$, such that $\sigma_\ast(b)\in \mathcal{M}(b)$. Such strategies are called {\em positional}.

Given a position $b_0$, the {\em game tree generated from $b_0$} is a tree defined recursively with root $b_0$ and the children of $b$ are exactly elements in $\mathcal{M}(b)$. Each branch of the tree is a play starting from $b_0$. Given a strategy $\sigma_\ast$ and a position $b_0$, the game tree generated by $b_0$ following $\sigma_\ast$ is defined recursively with root $b_0$, and if $b\in B_\ast$ it only has one child $\sigma_\ast(b)$, otherwise the children of $b$ are exactly $\mathcal{M}(b)$. Each branch of this tree is a play starting from $b_0$ according to $\sigma_\ast$, and $\sigma$ is winning for $\ast$ from $b_0$, if all the branches of this tree are winning plays for $\ast$. 

Finally, a game is called \emph{determined} if every position has a winning strategy either for $\forall$ or for $\exists$. A game is called \emph{positionally determined} if it's determined by positional strategies. Using backwards induction on game trees it can be shown that finite board games, i.e.\ games that all the branches of the game tree from any point has finite length, are determined.  

\subsection{Unfolding game on power-set  algebras} \label{ssec:unfolding game power-set}
In this section we recall the unfolding game defined on power-set  algebras that provides an alternative way to obtain the greatest and the least fixed points of a monotone map on a  power-set algebra. Let $\mathcal{P}(S)$ be a power-set algebra and $f$ be a monotone map on it. Consider the following game board.
\vspace{-0.2 cm}
\begin{center}
    \begin{tabular}{|c|c|c|}
    \hline
      \textbf{state}& \textbf{player} & \textbf{admissible
         moves} \\
         \hline
         $a \in S$  & $\exists$ & $\{A \subseteq S \mid a \in f(A)\}$\\
         $A \subseteq S$ & $\forall$ & $A$\\
         \hline
     \end{tabular}
\end{center}
\vspace{-0.2 cm}
We use $G_\exists^f(a)$  and $G_\forall^f(a)$
to denote games on the above board with the starting position $a$ and the winning condition that infinite plays are won by $\exists$ and $\forall$, respectively. For the finite play, the player who has no admissible moves loses in both cases. The following lemma says that the unfolding game can be used to describe the least and the greatest fixed points of $f$.

\begin{theorem}[Theorem 3.14, \cite{venema2008lectures}]\label{lem:unfolding game power-set}
The set $\{a\mid \exists\mbox{ has a winning strategy for } G_\exists^f(a)\}$ is the greatest fixed points of $f$, and the set $\{a\mid \exists \mbox{ has a winning strategy for } G_\forall^f(a)\}$ is the least fixed points of $f$. 
\end{theorem}

\subsection{Classical modal $\mu$-calculus} \label{ssec:modal_mu-calculus}
In this section, we give some  preliminaries on modal $\mu$-calculus. For more details on  modal $\mu$-calculus, we refer to \cite{venema2008lectures}. Let $\mathrm{P},\mathrm{X}$ be (countable) sets of atomic propositions and variables, respectively. The language of the classical modal $\mu$-calculus $\mathcal{L}_c$ is given by the following recursion. 

{{ \centering
$\phi::= \, p \mid \neg p\mid x\mid \bot \mid \top \mid \phi \vee \phi \mid \phi \wedge \phi \mid \Box \phi \mid \Diamond \phi \mid \mu x.\phi\mid \nu x.\phi, 
$
\par
}}
\smallskip
where $p\in \mathrm{P}$ and $x\in X$.
We use $\eta x. \phi$ to denote a formula which is of the form $\mu x.\phi$ or $ \nu x.\phi$. A variable $x$ is said to be a bound variable in a formula $\phi$ if it occurs under the scope of a fixed point operator $\mu x$ or $\nu x$, and free otherwise. Given a Kripke model $M = (W, R, V)$, and  a formula $\phi$  in which a free variable  $x$ occurs positively, formula $\phi$ defines a function  $\phi^+$  on the complex algebra of  frame $(W,R)$. Then $M,w\Vdash\mu x.\phi$  (resp.~$M,w\Vdash\nu x.\phi$) exactly when $w$ belongs to  the least  (resp.~greatest) fixed points  respectively of the map $\phi^+$. It is well known that the variables of any modal $\mu$-calculus formula can be rewritten in a way that no variable occurs as both {\em free variable} and {\em bound variable}. Moreover, every bound variable occurs in the scope of only one fixed point operator. Thus, for any bound variable $x$ and any modal $\mu$-calculus formula $\phi$, there exists an unique formula $\psi_x$ such that $\eta x. \psi_x$ is a sub-formula of $\phi$. We define a dependency order on the bound variables of a formula as follows: for any bound variables $x$ and $y$ appearing in the formula $\varphi$, $x \leq_\varphi y$ iff $y$ occurs freely in $\psi_x$.

The semantics for the language of the modal $\mu$-calculus on a Kripke model $M=(W,R,V)$ is given by the board game $G(M)$ described in the table below. Positions are of the form $(w,\varphi)$, where $\phi$ is a formula and  $w\in W$,

{\scriptsize
{{\centering
    \begin{tabular}{|c|c|c|c|c|c|}
    \hline
         \textbf{state}& \textbf{player} & \textbf{admissible
         moves}& \textbf{state}& \textbf{player} & \textbf{admissible
         moves} \\
         \hline
         $(w, \phi_1 \vee \phi_2)$ & $\exists$ & $\{(w,\phi_1), (w,\phi_2)\}$&
         $(w, \phi_1 \wedge \phi_2)$ & $\forall$ & $\{(w,\phi_1), (w,\phi_2)\}$\\ 
        $(w, p)$, where $V, w \vdash p$ & $\forall$ & $\emptyset$&
        $(w, p)$, where $V, w\not\vdash p$ & $\exists$ & $\emptyset$\\    
        $(w, \top)$ & $\forall$ & $\emptyset$&
        $(w, \bot)$ & $\exists$ & $\emptyset$\\
        $(w, \Diamond \phi)$ & $\exists$ & $\{(w', \phi) \mid w R w'\}$&
        $(w, \Box\phi)$ &  $\forall$ & $\{(w', \phi) \mid w R w'\}$\\
        $(w, \eta x.  \phi_x)$ & - & $\{(w, \phi_x)\}$&
         $(w, x)$ & - & $\{(w, \eta x. \phi_x)\}$\\
        \hline
    \end{tabular}
\par}}
}

\noindent where $\eta= \mu$ or $\eta =\nu$. For a finite play, a player wins if their opponent gets stuck (i.e. there is no admissible move available). An infinite play with starting position $(w,\varphi)$ is won by $\exists$ if the $\leq_\varphi$-maximal variable which repeats infinitely often is $\nu$-variable and by $\forall$ otherwise\footnote{Existence of such a maximal variable is easy to show: if $y$ and $x$ repeat infinitely often in a play and $x\nleq_\varphi y$, then either $y$ appears bounded in $\psi_x$, i.e.~, $y\leq_\varphi x$, or $y$ doesn't appear in $\psi_x$, in which case for $y$ to repeat again after a position of the form $(w,x)$, some variable $z$ that is free in $\psi_x$ must appear such that $z$ is also free in $\psi_y$. That is, $y \leq_{\varphi} z$, and if $y$ is repeated infinitely often so is $z$. See also \cite{venema2008lectures}.}. The following theorem shows that the above game is adequate for modal $\mu$-calculus.
\begin{theorem}[Theorem 3.30, \cite{venema2008lectures}]\label{lem:game modal mu-calculus}
Given a Kripke model $M$, $M,w\Vdash\phi$ if and only if $\exists$ has a winning strategy for $G(M)$ from position $(w, \phi)$. 
\end{theorem}

\subsection{Basic non-distributive modal logic and  polarity-based semantics}\label{sssec:basic non-distributive modal logic} 
In this section, we provide the language and semantics of the basic non-distributive modal logic. This logic is part of a family of lattice-based  logics, sometimes referred to as {\em LE-logics} (cf.~\cite{conradie2019algorithmic}), which have been studied in the context of a research program on the logical foundations of categorization theory \cite{conradie2016categories,conradie2017toward,conradie2021rough,conradie2020non}.
Let $\mathrm{P}$ be a (countable) set of atoms. The language $\mathcal{L}$ is defined as follows:

\footnotesize{
{{
\centering
$ \varphi::= \bot \mid \top \mid p \mid  \varphi \wedge \varphi \mid \varphi \vee \varphi \mid \Box \varphi \mid \Diamond \varphi$, \quad\quad\quad with $p$ ranging over $P$.
\par
}}
 \smallskip
}
The {\em basic}, or {\em minimal normal} $\mathcal{L}$-{\em logic} is a set $\mathbf{L}$ of sequents $\phi\vdash\psi$,  with $\phi,\psi\in\mathcal{L}$, containing the following axioms:
\smallskip

\footnotesize{
{{
\centering
\begin{tabular}{ccccccccccccc}
     $p \vdash p$ & \qquad & $\bot \vdash p$ & \qquad & $p \vdash p \vee q$ & \qquad & $p \wedge q \vdash p$ & \qquad & $\top \vdash \Box\top$ & \qquad & $\Box p \wedge \Box q \vdash \Box(p \wedge q)$
     \\
     & \qquad & $p \vdash \top$ & \qquad & $q \vdash p \vee q$ & \qquad & $p \wedge q \vdash q$ &\qquad &  $\Diamond\bot \vdash \bot$ & \qquad & $\Diamond(p \vee q) \vdash \Diamond p \vee \Diamond q$\\
\end{tabular}
\par}}
}
\smallskip
\noindent and closed under the following inference rules:
\[
			\frac{\phi\vdash \chi\quad \chi\vdash \psi}{\phi\vdash \psi}
			\ \ 
			\frac{\phi\vdash \psi}{\phi\left(\chi/p\right)\vdash\psi\left(\chi/p\right)}
			\ \ 
			\frac{\chi\vdash\phi\quad \chi\vdash\psi}{\chi\vdash \phi\wedge\psi}
			\ \ 
			\frac{\phi\vdash\chi\quad \psi\vdash\chi}{\phi\vee\psi\vdash\chi}
\ \ 
			\frac{\phi\vdash\psi}{\Box \phi\vdash \Box \psi}
\ \ 
\frac{\phi\vdash\psi}{\Diamond \phi\vdash \Diamond \psi}
\]
\paragraph{\bf Relational semantics.}
The following notation,  notions and facts are from \cite{conradie2021rough,conradie2020non}.
For any binary relation $T\subseteq U\times V$, and any $U'\subseteq U$  and $V'\subseteq V$,  we let

{{\centering
$T^{(1)}[U']:=\{v\mid \forall u(u\in U'\Rightarrow uTv) \}, \quad\quad T^{(0)}[V']:=\{u\mid \forall v(v\in V'\Rightarrow uTv) \}$
\par
}}
\smallskip 
\noindent A {\em polarity} or {\em formal context} (cf.~\cite{ganter2012formal}) is a tuple $\mathbb{P} =(G,M,I)$, where $G$ and $M$ are sets, and $I \subseteq G \times M$ is a binary relation. 
Intuitively, formal contexts can be understood as the abstract representations of databases \cite{ganter2012formal}, where  $G$ and  $M$ are collections of {\em objects} and {\em features} respectively, and the object $g \in G$ has the feature $m \in M$ whenever $gIm$.

As is well known, for every formal context $\mathbb{P} = (G, M, I)$, the pair of maps 

 {{\centering
 $(\cdot)^\uparrow: \mathcal{P}(G)\to \mathcal{P}(M)\quad \mbox{ and } \quad(\cdot)^\downarrow: \mathcal{P}(M)\to \mathcal{P}(G),$
\par}}
\smallskip 
\noindent  defined by the assignments $B^\uparrow:= I^{(1)}[B]$ and $Y^\downarrow := I^{(0)}[Y]$ where $B\subseteq G$ and $Y\subseteq M$,  form a Galois connection, and hence induce the closure operators $(\cdot)^{\uparrow\downarrow}$ and $(\cdot)^{\downarrow\uparrow}$ on $\mathcal{P}(G)$ and on $\mathcal{P}(M)$ respectively. The fixed points of $(\cdot)^{\uparrow\downarrow}$ and $(\cdot)^{\downarrow\uparrow}$ are  the {\em Galois-stable} sets.
A {\em formal concept} of a polarity $\mathbb{P}=(G,M,I)$ is a tuple $c=(B,Y)$ such that $B\subseteq G$ and $Y\subseteq M$, and $B = Y^\downarrow$ and $Y = B^\uparrow$.  The set $B$ (resp.~$Y$) is  the {\em extension} (resp.~the {\em intension}) of $c$ and is denoted by $\val{c}$  (resp.~$\descr{c}$). 
It is well known 
 that the set  of formal concepts of a polarity $\mathbb{P}$,  with the order defined by

{{\centering
$c_1 \leq c_2 \quad \text{iff} \quad \val{c_1} \subseteq \val{c_2} \quad \text{iff} \quad \descr{c_2} \subseteq \descr{c_1}$,
\par
}}

\noindent forms a complete lattice $\mathbb{P}^+$ join-generated (resp.~meet-generated) by the set $\{\textbf{g} = (g^{\uparrow\downarrow}, g^\uparrow)\}$ (resp.~$\{\textbf{m} = (m^{\downarrow}, m^{\downarrow\uparrow})\})$, namely  the {\em concept lattice} of $\mathbb{P}$. {\em Brikhoff's representation theorem} tells us that the converse also holds, i.e.\ for any complete lattice $\mathbb{L}$, there exists a polarity $\mathbb{P}$ such that $\mathbb{P}^+ \cong \mathbb{L}$.

For the language $\mathcal{L}$ defined above, an {\em enriched formal $\mathcal{L}$-context} is a tuple $\mathbb{F} =(\mathbb{P}, {R}_\Box,{R}_\Diamond)$, where $R_\Box  \subseteq G \times M  $ and $R_\Diamond \subseteq M \times G$  are  {\em $I$-compatible} relations, that is, for all  $g \in G$, and $m \in M$, the sets $R_\Box^{(0)}[m]$,  $R_\Box^{(1)}[g]$, $R_\Diamond^{(0)}[g]$,  and $R_\Diamond^{(1)}[m]$  are Galois-stable. For any $c \in \mathbb{P}^+$,

{{
\centering
      $[R_\Box] c =(R_\Box^{(0)}[\descr{c}], I^{(1)}[R_\Box^{(0)}[\descr{c}]]) \quad \text{and} \quad \langle R_\Diamond\rangle  c =( I^{(0)}[R_\Diamond^{(0)}[\val{c}]], R_\Diamond^{(0)}[\val{c}]).$
\par
}}

We refer to the algebra $\mathbb{F}^+=(\mathbb{P}^+, [R_\Box], \langle R_\Diamond \rangle )$ as the {\em complex algebra} of  $\mathbb{F}$.  A {\em valuation} on such an $\mathbb{F}$
is a map $V\colon\mathrm{P}\to \mathbb{P}^+$. For each  $p\in \mathrm{P}$, we let  $\val{p} := \val{V(p)}$ (resp.~$\descr{p} := \descr{V(p)}$) denote the extension (resp.\ intension) of the concept $V(p)$.  

A {\em model} is a tuple $\mathbb{M} = (\mathbb{F}, V)$, where $\mathbb{F} = (\mathbb{P}, {R}_{\Box},{R}_{\Diamond})$ is an enriched formal context and $V$ is a  valuation on $\mathbb{F}$. For every $\phi\in \mathcal{L}$, we let  $\val{\phi}_\mathbb{M} := \val{V(\phi)}$ (resp.~$\descr{\phi}_\mathbb{M} := \descr{V(\phi)}$) denotes the extension (resp.\ intension) of the interpretation of $\phi$ under the homomorphic extension of $V$.   The  `forcing' relations $\Vdash$ and $\succ$ are recursively defined: 
\smallskip

{{\centering 
\scriptsize
\begin{tabular}{l@{\hspace{1em}}l@{\hspace{2em}}l@{\hspace{1em}}l}
$\mathbb{M}, g \Vdash p$ & iff $g\in \val{p}_{\mathbb{M}}$ &
$\mathbb{M}, m \succ p$ & iff $m\in \descr{p}_{\mathbb{M}}$ \\
$\mathbb{M}, g \Vdash\top$ & always &
$\mathbb{M}, m \succ \top$ & iff   $g  I m$ for all $g\in G$\\
$\mathbb{M}, m \succ  \bot$ & always &
$\mathbb{M}, g \Vdash \bot $ & iff $g I m$ for all $m\in M$\\
$\mathbb{M}, g \Vdash \phi\wedge \psi$ & iff $\mathbb{M}, g \Vdash \phi$ and $\mathbb{M}, g \Vdash  \psi$ & 
$\mathbb{M}, m \succ \phi\wedge \psi$ & iff $(\forall g\in G)$ $(\mathbb{M}, g \Vdash \phi\wedge \psi \Rightarrow g I m)$
\\
$\mathbb{M}, m \succ \phi\vee \psi$ & iff  $\mathbb{M}, m \succ \phi$ and $\mathbb{M}, m \succ  \psi$ & 
$\mathbb{M}, g \Vdash \phi\vee \psi$ & iff $(\forall m\in M)$ $(\mathbb{M}, m \succ \phi\vee \psi \Rightarrow g I m )$.
\end{tabular}
\par}}
\smallskip

\noindent As to the interpretation of modal formulas:
\smallskip

{{\centering
\scriptsize
\begin{tabular}{llcll}
$\mathbb{M}, g \Vdash \Box\phi$ &  iff $(\forall m\in M)(\mathbb{M}, m \succ \phi \Rightarrow g R_\Box m)$ & \quad\quad &
$\mathbb{M}, m \succ \Box\phi$ &  iff $(\forall g\in G)(\mathbb{M}, g \Vdash \Box\phi \Rightarrow g I m)$\\
$\mathbb{M}, m \succ \Diamond\phi$ &  iff $(\forall g\in G)(\mathbb{M}, g \Vdash \phi \Rightarrow m R_\Diamond g)$ &&
$\mathbb{M}, g \Vdash \Diamond\phi$ & iff $(\forall m\in M)(\mathbb{M}, m \succ \Diamond\phi \Rightarrow g I m)$  \\
\end{tabular}
\par}}
\smallskip

\noindent The definition above ensures that, for any $\mathcal{L}$-formula $\varphi$,

{{\small\centering
$\mathbb{M}, g \Vdash \phi$  iff  $g\in \val{\phi}_{\mathbb{M}}$, \quad  and \quad$\mathbb{M},m \succ \phi$  iff  $m\in \descr{\phi}_{\mathbb{M}}$. \par}}

{{\small\centering
$\mathbb{M}\models \phi\vdash \psi$ \quad iff \quad $\val{\phi}_{\mathbb{M}}\subseteq \val{\psi}_{\mathbb{M}}$\quad  iff  \quad  $\descr{\psi}_{\mathbb{M}}\subseteq \descr{\phi}_{\mathbb{M}}$. 
\par}}
\noindent Hence, for any $\phi$, $g \Vdash \phi$ (resp.~$m \succ \phi$) iff $g$ (resp.~$m$) lies in the  extension (resp.~intension) of the concept $V(\phi)$.  The interpretation of the operators $\Box$ and $\Diamond$  is motivated by algebraic properties and  duality theory for modal operators on lattices (cf.~\cite[Section 3]{conradie2020non} for an expanded discussion).
The following theorem shows that the polarity-based semantics defined above is complete w.r.t.~the non-distributive modal logic.

\begin{theorem}[Proposition 3, \cite{conradie2017toward}]
For any  $\mathcal{L}$-sequent $\phi \vdash \psi$ not provable in the lattice-based modal logic, there exists a polarity-based model $\mathbb{M}$ such that $\mathbb{M} \not \models \phi\vdash \psi$.
\end{theorem}

\section{Game semantics for non-distributive modal logic} \label{sec:Game semantics for non-distributive modal logic}
In this section, we develop the game semantics for the non-distributive modal logic $\mathbf{L}$. An alternative game semantics presentation can be found in \cite{hartonas2019game}. 

In classical modal logic, the game semantics, as is shown in Lemma \ref{lem:game modal mu-calculus}, takes the form of an ``examination'' (which also provides the naming convention of the two players). $\forall$ is asking $\exists$ to show the truthfulness of a formula in a possible world in a Kripke model. This results in an asymmetry on the moves. Player $\exists$ moves when the formula under question is defined existentially, while $\forall$ moves when it is defined universally. As we saw in Section \ref{sssec:basic non-distributive modal logic}, the polarity semantics for lattice-based logics has two forcing relations, $\Vdash$ and $\succ$, and all formulas are recursively defined via universal statements. Hence, the game we define takes a different form. The two players stand against each other as equals. Player $\exists$ is responsible for showing that $g\Vdash\varphi$ and only moves on positions of the form $(m,\varphi)$, while $\forall$ is responsible for showing that $m\succ\varphi$ and only moves on positions of the form $(g,\varphi)$. Concretely, given an enriched polarity $(\mathbb{P},R_\Diamond,R_\Box)$, where $\mathbb{P}= (G,M,I)$, the game board is defined by the following table:

\vspace{-0.3 cm}
{\scriptsize
\begin{center}\setlength{\tabcolsep}{0pt}
    \begin{tabular}{|c|c|c|c|c|c|}
    \hline
         \textbf{state}& \textbf{player} & \textbf{admissible
         moves}  &   \textbf{state}& \textbf{player} & \textbf{admissible
         moves} \\
         \hline
        $(g, p)$, where $V, g \Vdash p$ & $\forall$ & $\emptyset$&
        $(g, p)$, where $V, g \not\Vdash p$ & $\forall$ & $\{(m,p) \mid gI^cm\}$\\
        $(m, p)$, where $V, m \succ p$ & $\exists$ & $\emptyset$&
        $(m, p)$, where $V, m \not\succ p$ & $\exists$ & $\{(g,p)\mid gI^cm\}$\\     
        $(g, \top)$ & $\forall$ & $\emptyset$&
        $(m, \bot)$ & $\exists$ & $\emptyset$\\
        $(g, \bot)$ & $\forall$ & $\{(m,\bot)\mid g I^c m\}$&
       $(m, \top)$ & $\exists$ & $\{(g,\top)\mid g I^c m\}$
        \\
        $(m, \phi_1 \vee \phi_2)$ & $\exists$ & $\{(m, \phi_1),  (m, \phi_2) \}$ &
        $(g, \phi_1 \vee \phi_2)$ & $\forall$ & $\{(m, \phi_1 \vee \phi_2) \mid g I^c m \}$\\
        $(g, \phi_1 \wedge \phi_2)$ & $\forall$ & $\{(g, \phi_1),  (g, \phi_2) \}$ &
        $(m, \phi_1 \wedge \phi_2)$ & $\exists$ & $\{(g, \phi_1 \wedge \phi_2) \mid g I^c m \}$\\
        $(m, \Diamond\phi)$ & $\exists$ & $\{(g, \phi) \mid m R_\Diamond^c g \}$&
        $(g, \Diamond \phi)$ & $\forall$ & $\{(m, \Diamond\phi) \mid g I^c m \}$\\
        $(g, \Box\phi)$ & $\forall$ & $\{(m, \phi) \mid g R_\Box^c m \}$&
        $(m, \Box \phi)$ & $\exists$ & $\{(g, \Box \phi) \mid g I^c m \}$\\
        \hline
    \end{tabular}
\end{center}
}
\vspace{-0.2 cm}
For finite plays, a player wins if the other player has no legitimate move. Infinite plays are uninteresting since they only arise when none of the player applies winning strategy. For the sake of completeness we stipulate that in infinite plays $\forall$ wins the game. We can dually define an adequate game semantics in which $\exists$ wins infinite plays.
The game is positional and determined because players can enforce a finite game to their benefit. 
The theorem below is an analogue of Theorem \ref{lem:game modal mu-calculus} and shows the correspondence between the evaluation game defined above and the polarity semantics. Given that the game is determined, it is enough to show the existence of winning strategies for one player. The proof of the theorem is given in Appendix \ref{appendix:proofs}. 

\begin{theorem}\label{lem:non_distributive_modal_logic_game}
For any valuation $V:\mathrm{P} \to \mathbb{L}$,  $\phi \in \mathbf{L}$, and  $g \in G$ (resp.~$m \in M$), $V, g \Vdash \phi$ (resp.~$V, m \not\succ \phi$) iff $\exists$ has a winning strategy for the game starting with the initial position $(g, \phi)$ (resp.~$(m,\phi)$). 
\end{theorem}

\section{Non-distributive modal $\mu$-calculus}\label{sec:Non_distributive_modal_mu_calculus}
In this section we expand the language of lattice-based modal logic with connectives that allow to describe the least and the greatest fixed points. We will start by introducing the language and its semantics via the complex algebra of the polarity. We refer to  the resulting logic as non-distributive modal $\mu$-calculus. Finally, we give examples of some scenarios in categorization theory which can be formalized in this  calculus.

\subsection{Syntax and semantics}\label{ssec:syntax and semantics}
Let $\mathrm{P}$ be a denumerable set of propositional variables.  The language of non-distributive modal $\mu$-calculus $\mathcal{L}_N$ is given by the following recursion. For the sake of our presentation we will introduce two sets of variables $X$ and $Y$ (in principle disjoint). Elements $x\in X$ (resp.~$y\in Y$) will be exclusively bounded by the greatest (resp.~least) fixed point operator. This choice is purely stylistic and allows for a clear presentation of the game semantics in Section \ref{ssce: gamesem}. 

{{
\centering
$\phi::= \, p   \mid x\mid y \mid \bot \mid \top \mid \phi \vee \phi \mid \phi \wedge \phi \mid\Box\phi\mid\Diamond\phi\mid \mu y.\phi\mid \nu x.\phi,$ 
\par
}}

\noindent where $p\in \mathrm{P}, x\in X$ and $y\in Y$. Again, we assume that the variables of any formula are renamed in such a way that no variable occurs as both free and bound variable and  every bound variable occurs under the scope of only one fixed point operator. Similar to modal $\mu$-calculus, given a formula $\varphi$ we can again define the dependency order on its bound variables,  $z_1 \leq_\varphi z_2$ iff $z_2$ occurs freely in $\psi_{z_1}$.

Given now an enriched polarity based on $\mathbb{P}$, and its concept lattice $\mathbb{P}^+$, any valuation $v$ on atomic propositions and variables extends to a valuation on every formula. Hence, given $v:\mathrm{P}\to\mathbb{P}$, we will treat $\overline{v}(\phi):\mathbb{P}^+\to\mathbb{P}^+$, as a function with input the variables of $\varphi$ and output the interpretation of $\varphi$ under that assignment of its variables. Then we interpret $\overline{v}(\mu y.\phi)$ as $\mathrm{lfp}(\overline{v}(\phi)(y))$ and $\overline{v}(\nu x.\phi)$ as $\mathrm{gfp}(\overline{v}(\phi)(x))$. Finally, we have: 

{{ \centering
 $g\Vdash \eta z.\phi$ \quad iff\quad $\textbf{g}\leq\overline{v}(\eta z.\phi)$ \quad and \quad $m\succ \eta z.\phi$ \quad iff\quad $\overline{v}(\eta z.\phi)\leq\textbf{m}$.
 \par
}}

\subsection{Fixed point operators in categorization}\label{ssec:Fixed point operators in categorization} 
In this section, we discuss some scenarios where the least and the greatest fixed point operators of some monotone operations on categories (or formal concepts) arise naturally. These examples show that the fixed point operators indeed play a crucial role in formal reasoning involving the dynamics of formal concepts, and they can often be formally modeled in the language of the non-distributive modal $\mu$-calculus. 

\fakeparagraph{Coalitions and issues.}
Consider a group of (finite) agents $G$,  and  a set of  (finite) issues $M$. Let $I \subseteq G \times M$, and $R_\Box \subseteq G \times M$ be relations  defined as $g I m$ iff the agent $g$  is definitely interested in the  issue $m$, and $g R_\Box m$ iff the agent $g$ is potentially interested  in the issue $m$, respectively. Therefore, $I \subseteq R_\Box$, implying  that for any concept $c$,  $c \leq \Box c$.  We consider  the coalitions that are represented by the concepts of the polarity $(G,M,I)$, where the extension and the intension of a concept describes the members and agenda (issues of their common interest) of the coalition, respectively. 

Suppose that the agents in a coalition $c=(B,Y)$ want to expand the coalition to gain more support.  Agents in $c$ know that $R_\Box^{0}[Y]$  is the set of agents potentially interested in their agenda. Thus,  these are  likely to be the easiest agents  for them to attract. To attract these members, they broaden their coalition from $c$ to $\Box c = (R_\Box^{0}[Y], I^{(1)}[R_\Box^{0}[Y]] $. That is, they decide to narrow their agenda to only consider issues  which are of interest  to all the members in $R_\Box^{0}[Y]$. As $\Box$ is a monotone operator, and the sets $G$ and $M$ are finite, this expansion process can be repeated recursively until it reaches a fixed point. This fixed point is given by the least fixed point of the operation $f(c)= c_0 \vee \Box c$, where $c_0$ is the initial coalition undergoing expansion. That is, the  coalition at the end of  the above expansion  process is given by $\mu c. (c_0 \vee \Box c)$. 

A dual   scenario  occurs when  a coalition tries to broaden its agenda, that is, to make the agenda sharper and to contract the coalition. Let $R_\Diamond\subseteq M\times G$ be a binary relation such that $m R_\Diamond g$ iff the agent $g$ is potentially interested in the issue $m$. Therefore,  $I \subseteq R_\Diamond^{-1}$, implying that  for any concept $c$,  $\Diamond c \leq c$. Suppose the agents in a coalition $c=(B,Y)$ decide to broaden their agenda by the following process. Firstly, they collect all the issues that all the members in $c$ are possibly interested in (i.e.~the set $R_\Diamond^{(0)}[B]$), and then only the members interested in  all the issues in this broadened set of issues remain in the coalition (i.e.~the set  $I^{(0)}[R_\Diamond^{(0)}[B]]$).   Thus, the coalition at the end of this broadening of agenda is $\Diamond c =(I^{(0)}[R_\Diamond^{(0)}[B]],R_\Diamond^{(0)}[B])$. This  procedure can be repeated recursively until we reach a fixed point n $\nu c.(c_0 \wedge \Diamond c)$, where $c_0$ is the initial coalition.

\fakeparagraph{Approximation and exact concepts.}
In \cite{conradie2021rough},  non-distributive modal logic is interpreted  as the logic  of rough concepts. Under this interpretation, the operators $\Box$ and $\Diamond$ are seen as the lower and upper approximation operators of a concept respectively.
We say that a concept $c$ is {\em exact from below} (resp.~{\em  above}) if its lower (resp.~upper) approximation is equal to itself. Then, for any concept $c_0$, the largest (resp.~the smallest) concept exact from below (resp.~above) contained in $c$ (resp.~containing $c$) is given by   $\nu c.(c_0 \wedge  \Box c)$ (resp.~$\mu c.(c_0 \vee  \Diamond c)$). 

\fakeparagraph{Social media community.}
Consider a social media platform which allows members to form communities using some rules.  Let $A$ be a set of members on the platform and $X$ be a set of features they might have. Let $I \subseteq A \times X$, $R_\Box \subseteq A \times X$ and $R_\Diamond \subseteq X \times A$ be the relations interpreted as follows: $a I x$ if $a$ has feature $x$, $x R_\Diamond a$ if $a$ thinks it is necessary for people in the 
community to have feature $x$ and, $a R_\Box x$ if everyone believes $a$ has feature $x$.  Suppose a  community functions using the following rules. Firstly, the set of features everyone believes members of community must have is collected. For community $c=(B,Y)$, this list is given by $R_\Diamond^{(0)}[B]$. Only the members who everyone believes have these features are allowed to stay in the community. Thus, after the application of this process, the resulting community is given by  $\Box \Diamond c= (R_\Box^{(0)}[R_\Diamond^{(0)}[B]], I^{(1)}[R_\Box^{(0)}[R_\Diamond^{(0)}[B]]])$. This process takes place periodically in this community. The greatest and the least fixed points of the operation $f(c)= \Box \Diamond  c$ in this scenario will then represent the largest and the smallest stable communities that can be formed with these rules. In the modal $\mu$-calculus they will be represented by concepts $\nu c.\Box\Diamond c$ and $\mu c.\Box\Diamond c$, respectively.

\section{Unfolding games on lattices and game semantics for non-distributive modal $\mu$-calculus}\label{sec:Game semantics for non-distributive modal mu-calculus}
The aim of this section is to provide a game semantics  for the non-distributive modal $\mu$-calculus. To this end, we   generalize the unfolding game defined in Section \ref{ssec:unfolding game power-set} to arbitrary complete lattices with monotone functions. We then  introduce the game semantics for the fixed point operators and prove its adequacy using the unfolding games.
\subsection{Fixed points on complete lattices and unfolding game}\label{ssec:unfolding on lattices}
 Let $\mathbb{L}$ be any  complete lattice and let $\jty(\mathbb{L})$ and $\mty(\mathbb{L})$  be sets of join and meet generators respectively (that is sets that join generate and meet generate the lattice respectively). We use $\nomj$ and $\cnomm$ to denote a join generator and a meet generator of  $\mathbb{L}$ respectively. Here we define games involving join generators. In Appendix \ref{app:B}, we present the dual games involving meet generators.

\paragraph{\bf Unfolding games with join generators.}\label{sssec:Unfolding games with join generators}
We consider the following game board:
\vspace{-0.5 cm}
\begin{center}
    \begin{tabular}{|c|c|c|}
    \hline
      \textbf{state}& \textbf{player} & \textbf{admissible
         moves} \\
         \hline
         $\nomj\in \jty(\mathbb{L})$  & $\exists$ & $\{S \subseteq\jty(\mathbb{L}) \mid \nomj \leq f(\bigvee S)\}$\\
         $S\subseteq\jty(\mathbb{L})$ & $\forall$ & $S$\\
         \hline
     \end{tabular}
\end{center}
\vspace{-0.2 cm}
First, we consider the game in which all infinite plays are won by $\exists$. We denote this game starting from $\nomj\in \jty(\mathbb{L})$ by $G^f_{\jty,\exists}(\nomj)$. We have:
\begin{lemma} \label{lem:unfolding_join_exists}
Let $W =\{ \nomj \in \jty(\mathbb{L}) \mid \exists \mbox{ has a winning strategy for } G^f_{\jty,\exists}(\nomj)\}$. Then $\bigvee W$ is the greatest fixed point of $f$. 
\end{lemma}
\begin{proof}
We show that $\bigvee W$ is a post-fixed point of $f$, namely that $\bigvee W \leq f(\bigvee W)$, or equivalently that $\nomj \leq f(\bigvee W)$ for any $\nomj \in W$.
Let $\nomj\in W$. Since $\exists$ has a winning strategy for $\nomj$, there is a $T \subseteq\jty(\mathbb{L})$ such that $\nomj \leq f(\bigvee T)$ and $\exists$ has a winning strategy for any move that $\forall$ can make, i.e., $T\subseteq W$. Hence $f(\bigvee T)\leq f(\bigvee W)$, and since $\nomj\leq f(\bigvee T)$, we have $\nomj\leq f(\bigvee W)$. 

As $\bigvee W$ is a post-fixed point of $f$, it follows that $\bigvee W \leq v$, where $v$ is the greatest fixed point of $f$. To show the converse inequality is equivalent to show that $\forall \nomi(\nomi \leq v \Rightarrow \nomi \leq \bigvee W)$; therefore it is sufficient to show that any $\nomi \leq v$ is in $W$. Let $\nomi$ be a nominal such that $\nomi \leq v$. A winning strategy in the game $G^f_{\jty,\exists}(\nomj)$ for $\exists$ consists in always picking $\{\nomj\in  \jty(\mathbb{L})\mid \nomj\leq v\}$: as $f(v) = v$, it will always be a possible choice for any subsequent choice of $\forall$. This concludes the proof. 
\end{proof}
Now consider a variation of the game where all infinite plays are won by $\forall$. We denote this game by $G^f_{\jty,\forall}(\nomj)$, where $\nomj \in \jty(\mathbb{L})$ is the starting position of the game. 

\begin{lemma}\label{lem:unfolding_join_forall}
Let $W=\{ \nomj \in \jty(\mathbb{L}) \mid \exists \mbox{ has a winning strategy for }G^f_{\jty,\forall}(\nomj)\}$. Then $\bigvee W$ is the least fixed point of $f$. 
\end{lemma}
\begin{proof}
We show that $\bigvee W$ is a pre-fixed point of $f$, namely that $f(\bigvee W) \leq\bigvee W$, or equivalently that $\nomj\leq\bigvee W$ for any $\nomj \leq f(\bigvee W)$. Let $\nomj \leq f(\bigvee W)$. At position $\nomj$, $\exists$ can play $W$. Then $\forall$ must choose a position $\nomi\in W$ to play, but this means $\exists$ has winning strategy at $\nomi$. Therefore, $\exists$ has a winning strategy at $\nomj$, i.e., $\nomj\in W$. Hence $\nomj\leq\bigvee W$.

As $\bigvee W$ is a pre-fixed point of $f$, it follows that $u\leq\bigvee W$, where $u$ is the least fixed point of $f$. To show the converse inequality, it is enough to show that $\nomj\leq u$ for any $\nomj\in W$. Assume not, i.e., there exists $\nomi_0\in W$ and $\nomi_0\nleq u=f(u)$. Because $\exists$ has a winning strategy at $\nomi_0$, there exist a $B_0\subseteq\jty(\mathbb{L})$ such that $\nomi_0\leq f(\bigvee B_0)$ and since $B_0$ is a winning move, $B_0\subseteq W$. Therefore, $f(\bigvee B_0)\nleq f(u)$. By the monotonicity of $f$, we have that $\bigvee B_0\nleq u$, i.e., $B_0\nsubseteq \{\nomj\mid \nomj\leq u\}$. So there exists a position $\nomi_1\in B_0$ and $\nomi_1\nleq u$, where when played by $\forall$, $\exists$ has a winning answer $B_1$. Continuing recursively we can define a sequence of winning positions $(\nomi_n)_{n\in\omega}$ and $(B_n)_{n\in\omega}$, such that $B_n\neq\varnothing$, $B_n$ is a winning move of $\exists$ to $\nomi_n$, and $\nomi_{n+1}\in B_n$. This means that the sequence $(\nomi_n,B_n)$ is an infinite play, according to a winning strategy of $\exists$, a contradiction since on infinite plays $\forall$ wins. This concludes the proof.
\end{proof}
In Appendix \ref{appendix:C}, we give more details on the choices made in generalizing unfolding games from power-set algebras to arbitrary complete lattices.  

\subsection{Game semantics for non-distributive modal $\mu$-calculus}\label{ssce: gamesem}
The game semantics for the non-distributive modal $\mu$-calculus is obtained by adding the following rules regarding the formulas containing fixed point operators to the game for non-distributive modal logic. In what follows we use $x$ to denote variables bounded by $\nu$, and $y$ to denote variables bounded by $\mu$.
\vspace{-0.2 cm}
\begin{center}
    \begin{tabular}{|c|c|c|c|c|c|}
    \hline
         \textbf{state}& \textbf{player} & \textbf{admissible
         moves} & \textbf{state}& \textbf{player} & \textbf{admissible
         moves}  \\
         \hline
         $(g, \nu x. \phi_x)$ & - &  $\{(g, \phi_x)\}$&
         $(g, x)$ & - &  $\{(g, \nu x. \phi_x )\}$\\
         $(m, \nu x. \phi_x)$ & $\exists$ &  $\{(g, \nu x. \phi_x) \mid g I^c m\}$&
         $(m, x)$ & $\exists$ &  $\{(g, x) \mid g I^c m\}$\\
         $(m, \mu y. \phi_y)$ & - &  $\{(m, \phi_y)\}$&
         $(m, y)$ & - &  $\{(m, \mu y. \phi_y )\}$\\
         $(g, \mu y. \phi_y)$ & $\forall$ &  $\{(m, \mu y. \phi_y) \mid g I^c m\}$&
         $(g, y)$ & $\forall$ &  $\{(m, y) \mid g I^c m\}$\\
         \hline
\end{tabular}
\end{center}
\vspace{-0.1 cm}
If the game ends in the finite number of moves, then the player who gets stuck in the end loses the game. In case, the game lasts infinite number of moves, if the maximal bound variable in the dependency order which gets repeated infinitely often in a  play is $\nu$ variable, then the game is won by $\exists$ and if it is $\mu$ variable then by  $\forall$. These games are parity games, and so are also determined \cite[5.22]{venema2008lectures}.


The theorem below shows that the game described above is adequate for non-distributive modal $\mu$-calculus; the proof can be found in Appendix \ref{appendix:proofs}. 
\begin{theorem}\label{thm:adequacy_non_distributive_modal_mu_calculus}
 For any valuation $V:\mathrm{P} \to \mathbb{L}$,  $\phi \in \mathcal{L}_N$, and  $g \in G$ (resp.~$m \in M$), $V,g \Vdash \phi$ (resp.~$V, m \not\succ \phi$) iff $\exists$ has a winning strategy for the game starting with the initial position $(g, \phi)$ (resp.~$(m,\phi)$). 
\end{theorem}

\begin{example}
    Let $\varphi$ be the formula $\nu x. (\Box x\wedge q)$, and $\mathbb{M}=(\mathbb{F}, V)$ be a model for non-distributive modal $\mu$-calculus, where $G=\{g_1,g_2\}$, $M=\{m_1,m_2,m_3\}$, $I=R_\Box=\{(g_1,m_1),(g_1,m_2),(g_2,m_3)\}$, $R_\Diamond = \emptyset$, and $V(q)=(\{g_1\}, \{m_1,m_2\})$. We demonstrate the evaluation game by showing:  $(1) \, \mathbb{M}, g_1\Vdash \varphi \quad \text{and} \quad (2)\, \mathbb{M},m_2\succ\varphi$. 

\noindent (1) We show that $\exists$ has a winning strategy for the evaluation game with the initial position $(g_1,\varphi)$. By the game rules, starting from $(g_1,\nu x.(\Box x\wedge q)$, the next step is $(g_1,\Box x\wedge q)$ chosen by $\forall$ or $\exists$. Then it's $\forall$'s turn and he can choose $(g_1,\Box x)$ or $(g_1,q)$. If $\forall$ chooses $(g_1,q)$, he will have no admissible moves for next step since $\mathbb{M},g_1\Vdash q$. If he chooses $(g_1,\Box x)$, then under the condition $gR^c_\Box m$, he can only move to $(m_3,x)$. In this position,  $\exists$ can move to $(g_1,x)$ by $g_1 I^c m_3$ which leads to the position  $(g_1,\varphi)$ again.  If the game goes infinitely, it will be winning for $\exists$ because $x$ is $\nu$-variable. Hence, no matter what $\forall$ chooses, $\exists$ always has a winning strategy with the initial position $(g_1,\varphi)$. By  Theorem~\ref{thm:adequacy_non_distributive_modal_mu_calculus}, $\mathbb{M},g_1 \Vdash \nu x.(\Box x\wedge q)$.

    \noindent (2) We show  that $\exists$ does not have a winning strategy in the evaluation game  with the initial position $(m_2,\varphi)$, i.e. the game starting from $(m_2,\varphi)$ is winning for $\forall$. By the game rules, from the state $(m_2,\nu x.(\Box x\wedge q))$, $\exists$ will have to move to $(g_2,\nu x.(\Box x\wedge q))$.  Then  $\forall$ or $\exists$  move to $(g_2,\Box x \wedge q)$. In this position,  $\forall$ can choose to move to  $(g_2,q)$. Since $\mathbb{M},g_2 \not\Vdash q$, $\forall$ can move to $(m_1,q)$ as $g_2 I^c m_1$. Since $\mathbb{M}, m_1\succ q$, $\exists$ does not have any move, and $\forall$ wins. By Theorem~\ref{thm:adequacy_non_distributive_modal_mu_calculus}, $\mathbb{M},m_2\succ\nu x.(\Box x\wedge q)$.

\end{example}

\section{Bisimulations on polarity-based models}\label{sec:bisimulations_polarity}
In this section, we generalize the notion of  bisimulations to the polarity-based models and show that the non-distributive modal $\mu$-calculus is bisimulation invariant.

\begin{definition}\label{def:simulation}
Let $\mathbb{M}_1=(\mathbb{F}_1, V_1)$ and $\mathbb{M}_2=(\mathbb{F}_2, V_2)$ be polarity-based models with 
$\mathbb{F}_i = (G_i,M_i,I_i, {R_\Box}_i, {R_{\Diamond}}_i)$ for $i=1,2$. A simulation from $\mathbb{M}_1$ to $\mathbb{M}_2$ is a pair of relations $(S,T)$ such that $S \subseteq G_1 \times G_2$ and $T \subseteq M_1 \times M_2$ and for any proposition $p$, $g_1\in G_1$, $g_2 \in G_2$, $m_1 \in M_1$, $m_2 \in M_2$, 
\begin{enumerate}[noitemsep, nolistsep]
   \item If $g_1 S g_2$, then if $ g_1 \in \val{V_1 (p) }$ then $ g_2 \in \val{V_2 (p) }$.
   \item If $m_1 T m_2$, then if $ m_2 \in \descr{V_2 (p) }$ then $ m_1 \in \descr{V_1 (p) }$.
   \item  If $g_1 S g_2$ and $g_2 {I_2^c} m_2$, then there exists $m_1 \in M_1$ such that  $g_1 {I_1^c} m_1$ and  $m_1 T m_2$.
   \item  If $m_1 T m_2$ and $g_1 {I_1}^c m_1$, then there exists $g_2 \in G_2$ such that  $g_2 {I_2}^c m_2$ and  $g_1 S g_2$.
   \item  If $g_1 S g_2$ and $g_2 {{R_\Box}_2}^c m_2$, then there is $m_1 \in M_1$ such that  $g_1 {{R_\Box}_1}^c m_1$ and  $m_1 T m_2$.
   \item  If $m_1 T m_2$ and $m_1 {{R_\Diamond}_1}^c g_1$, then there is $g_2 \in G_2$ such that  $m_2 {{R_\Diamond}_1}^c g_2$ and  $g_1 S g_2$. 
\end{enumerate}
\end{definition}
The conditions 3,4,5, and 6 are counterparts of ``zig'' or ``forth'' conditions in the definition of bisimulations on Kripke frames.

\begin{theorem}\label{thm:simulation invariance}
Let $\mathbb{M}_1=(\mathbb{F}_1, V_1)$ and $\mathbb{M}_2=(\mathbb{F}_2, V_2)$ be the polarity-based models with 
$\mathbb{F}_i =  (G_i,M_i, I_i, {R_\Box}_i, {R_{\Diamond}}_i)$  for $i=1,2$. Let $(S,T)$ be a simulation from $\mathbb{M}_1$ to $\mathbb{M}_2$. Then for any formula $\phi$,
\begin{enumerate}[noitemsep, nolistsep]
    \item $g_1 S g_2$ implies if $\mathbb{M}_1 ,g_1 \Vdash \phi$ then  $\mathbb{M}_2 ,g_2 \Vdash \phi$. 
    \item $m_1 T m_2$ implies if $\mathbb{M}_2 ,m_2 \succ \phi$ then $\mathbb{M}_1 ,m_1 \succ \phi$.
\end{enumerate}
\end{theorem}
The proof  proceeds by showing that the winning strategy of $\exists$ in some position $(g_1, \phi)$ (resp.~$(m_1, \phi)$) can be used to define a winning strategy for $\exists$ in position $(g_2, \phi)$ (resp.~$(m_2, \phi)$). See Appendix \ref{appendix:proofs} for the details. 
\begin{remark}
In classical logic,  if $\forall \phi (\mathbb{M}, w \models \phi \implies \mathbb{M}, w' \models \phi)$ holds,  then by negating the formulas  the reverse implication $\forall \phi (\mathbb{M}, w' \models \phi \implies \mathbb{M}, w \models \phi)$ also holds. However, in the non-distributive case negation is not in the language. Thus, the existence of a simulation between two points needs not to imply the full semantic equivalence. 
\end{remark}

\begin{definition}\label{def:bisimilar}
Let $\mathbb{M}_1=(\mathbb{F}_1, V_1)$ and $\mathbb{M}_2=(\mathbb{F}_2, V_2)$ be the polarity-based models with $\mathbb{F}_i =  (G_i,M_i, I_i, {R_\Box}_i, {R_{\Diamond}}_i)$  for $i=1,2$. The elements $g_1$ and $g_2$(resp. $m_1$ and $m_2$) are bisimilar if there exists a simulation $(S_1, T_1)$ from $\mathbb{M}_1$ to $\mathbb{M}_2$ and a simulation $(S_2, T_2)$ from $\mathbb{M}_2$ to $\mathbb{M}_1$ such that $g_1 S_1 g_2$ and $g_2 S_2 g_1$(resp. $m_1 T_1 m_2$ and $m_2 T_2 m_1$).
\end{definition}

The following corollary is immediate from   Definition \ref{def:bisimilar} and Theorem \ref{thm:simulation invariance}. 
\begin{corollary}
 Let $\mathbb{M}_1=(\mathbb{F}_1, V_1)$ and $\mathbb{M}_2=(\mathbb{F}_2, V_2)$ be the polarity-based models with  $\mathbb{F}_i =  (G_i,M_i, I_i, {R_\Box}_i, {R_{\Diamond}}_i)$  for $i=1,2$. Then for any formula $\phi$,
\begin{enumerate} [noitemsep, nolistsep]
    \item for any $g_i\in G_i$, $g_1$ and $g_2$ are bisimilar implies $\mathbb{M}_1 ,g_1 \Vdash \phi$ iff  $\mathbb{M}_2 ,g_2 \Vdash \phi$. 
    \item for any $m_i\in M_i$, $m_1$ and $m_2$ are bisimilar implies $\mathbb{M}_2 ,m_2 \succ \phi$ iff $\mathbb{M}_1 ,m_1 \succ \phi$.
\end{enumerate}   
\end{corollary}

\section{Conclusion and future direction} \label{sec:Conclusion}
In this paper, we have generalized modal $\mu$-calculus to complete lattices. We show that unfolding games can be generalized from power-set algebras to complete lattices and they can be used to determine the greatest and the least fixed points of any monotone operator in a complete lattice. We then define  an evaluation game for non-distributive modal $\mu$-calculus in polarity-based semantics and show adequacy of this game. Finally, we define a notion of bisimulation on polarity-based models and show the invariance  of  non-distributive modal $\mu$-calculus under it. This paper provides several interesting directions for future research.

\fakeparagraph{Model theory for non-distributive modal $\mu$-calculus.} It would be interesting to study if the results like the Hennessy-Milner theorem, or the Van Benthem characterization theorem can be proved for non-distributive logics using the  notion of bisimulation  defined here. 
Another interesting model theoretic property to study for  non-distributive modal $\mu$-calculus is finite model property. 

\fakeparagraph{Proof theory for non-distributive modal $\mu$-calculus.} In future work, we plan to develop a proof system for non-distributive modal $\mu$-calculus. Completeness of modal $\mu$-calculus with respect to Kozen's axiomatization is an important result in the theory of modal $\mu$-calculus. We would like to explore the possibility of providing similar axiomatization in non-distributive setting.
\fakeparagraph{Modelling categorization dynamics and applications.}
As mentioned in Section~\ref{sec:intro}, the fixed point operators may  arise in categorization in many social and computational dynamic processes. Some such examples are discussed in Section \ref{ssec:Fixed point operators in categorization}. We intend to further explore the possibilities of  modelling such scenarios in the non-distributive modal $\mu$-calculus. 
\bibliography{ref}
\bibliographystyle{plain}
\appendix
\section{Proofs}\label{appendix:proofs}

\paragraph{\bf Proof of Theorem \ref{lem:non_distributive_modal_logic_game}.}
    We proceed via induction on the length of $\phi$. For the base case, assume that $V,g\Vdash p$. By the rules of the game, $(g,p)$ is a position in which  $\forall$ has no moves and hence loses, so this is the winning position for $\exists$. Likewise $V,m\succ p$ is winning for $\forall$. Assume $V,g\not\Vdash p$, then there exists some $m$, such that $V, m\succ p$ and $gI^cm$. Then picking $(m,p)$ is a winning strategy for $\forall$, and hence $\exists$ doesn't have a winning strategy.  The cases for $V,m\not\succ p$, and when $\varphi$ is $\top$ or $\bot$ are analogous. 

    
    Let $\phi$ be $\phi_1\vee\phi_2$. Assume $V,m\not\succ\phi_1\vee\phi_2$. Then by  definition, $m\not\succ\phi_1$ or $m\not\succ\phi_2$, and by the induction hypothesis, $\exists$ has a winning strategy $\sigma_1$  at $(m,\phi_1)$ or  a winning strategy $\sigma_2$  at $(m,\phi_2)$. Assume w.l.o.g.\ the first is the case. Then the strategy   in which  she picks $(m,\phi_1)$ at $(m,\phi_1\vee\phi_2)$  and follows $\sigma_1$ is a winning strategy for $\exists$. Conversely, assume $\exists$ has a winning strategy at $(m,\phi_1\vee\phi_2)$, then it means $\exists$ has a winning strategy $\sigma_1$ at $(m,\phi_1)$ or  a winning strategy $\sigma_2$ at $(m,\phi_2)$. By induction hypothesis, $m\not\succ\phi_1$ or $m\not\succ\phi_2$, which by definition implies $m\not\succ\phi_1\vee\phi_2$.

    Assume $V,g\Vdash\phi_1\vee\phi_2$, then by definition, for all $m\in M$ if $gI^c m$ then $m\not\succ\phi_1$ or $m\not\succ\phi_2$, following that $m\not\succ\phi_1\vee\phi_2$.  Hence, by the discussion in the previous paragraph, $\exists$ has a winning strategy $\sigma_m$ at $(m,\phi_1\vee\phi_2)$ for all $m\in M$ such that $gI^c m$. This means at $(g,\phi_1\vee\phi_2)$, no matter what $\forall$ chooses, $\exists$ always has a winning strategy. Conversely, assume that $\exists$ has a winning strategy $\sigma$ at $(g,\phi_1\lor\phi_2)$, which means that for every $m\in M$ such that $gI^c m$, $\exists$ has a winning strategy at $(m,\phi_1\lor\phi_2)$. By the discussion above, this means $V,m\not\succ \phi_1\lor\phi_2$ for all such $m$. That is, if $V,m\succ\phi_1\lor\phi_2$, then $gIm$, which by definition means that $V,g\Vdash\phi_1\lor\phi_2$.

    Let $\phi$ be $\phi_1\wedge\phi_2$. The case for $V,g\Vdash\phi_1\wedge\phi_2$ (resp.~$V,m\not\succ\phi_1\wedge\phi_2$) is verbatim the same as the case for $V,m\not\succ\phi_1\vee\phi_2$ (resp.~$V,g\Vdash\phi_1\vee\phi_2$).
    


    Let $\phi$ be $\Diamond\psi$. Assume $V,m\not\succ\Diamond\psi$, then by definition, there exists $g_0\in G$ such that $g_0\Vdash\psi$ but $mR^c_\Diamond g_0$. By induction hypothesis, $\exists$ has a winning strategy $\sigma$ at such $(g_0,\psi)$. Thus, she has a winning strategy at $(m,\Diamond\psi)$, which is picking $(g_0,\psi)$ and following $\sigma$ afterwards. Conversely, assume that $\exists$ has a winning strategy at $(m,\Diamond\psi)$. This means there exists some $g\in G$ such that $m R^c_\Diamond g$ and she has a winning strategy at $(g,\psi)$. By the induction hypothesis $V,g\Vdash\psi$, while $m R^c_\Diamond g$, and so $V,m\not\succ\Diamond\psi$.

    Assume $g\Vdash\Diamond\psi$, then by definition, for all $m\in M$, if $gI^c m$ then $m\not\succ\Diamond\psi$. By the discussion above, $\exists$ has a winning strategy at $(m,\Diamond\psi)$, for all such $m$. Thus no matter what $\forall$ plays $\exists$ has a winning strategy, which means $\exists$ has a winning strategy at $(g,\Diamond\psi)$. Conversely, assume $\exists$ has a winning strategy at $(g,\Diamond\psi)$. Then for every $m\in M$ such that $gI^cm$, $\exists$ has a winning strategy at $(m,\Diamond\psi)$. By the above paragraph, this means that $V,m\not\succ\Diamond\psi$ for all such $m$. That is, if $V,m\succ\Diamond\psi$ then $gIm$ which, by definition, means that $g\Vdash\Diamond\psi$.

    Let $\phi$ be $\Box\psi$. The case for $V,g\Vdash \Box \psi$ (resp.~$V,m\not\succ \Box \psi$) is analogous  to  the case for $V,m\not\succ \Diamond \psi$ (resp.~$V,g\Vdash \Diamond \psi$). This concludes the proof. 
    

\paragraph{\bf Proof of Theorem \ref{thm:adequacy_non_distributive_modal_mu_calculus}.}
To simplify notation, in what follows, $\phi_x$ will be treated both as a formula and as the map on the concept lattice $\mathbb{L}=\mathbb{P}^+$, given by $\overline{v}(\varphi_x)$. That is, $\phi_x: \mathbb{L} \to  \mathbb{L}$ given by $\phi_x(c)= (\val{\phi(x)}, \descr{\phi(x)})$ where $(\val{x}, \descr{x})$ is set to $c$. 

The proof is by induction on the complexity of the formula $\phi$. The base case and inductive case for propositional and modal operators were already discussed in the proof of the Theorem \ref{lem:non_distributive_modal_logic_game}. Thus, to complete the proof we only need to show the inductive step for the fixed point operators. Let $\mathbb{L}=\mathbb{P}^+$ where $\mathbb{P} = (G,M,I)$. Suppose $\phi:= \nu x . \psi_x$ and let $c$ be the greatest fixed point of the function $ \psi_x$.

($\Rightarrow$) Let $V, g \Vdash \phi$. By definition, we have $g \in \val{c}$.  By Lemma \ref{lem:unfolding_join_exists}, the game $G^{\psi_x}_{\jty,\exists}(g^{\uparrow\downarrow})$ is winning for $\exists$, and in particular a winning strategy is always playing $c$. 
Let $\mathrm{P}' = \mathrm{P} \cup \{x\}$ and let $V'$ be extension of $V$ to $\mathrm{P}'$ with  $V'(x)=(\val{c},\descr{c})$. By the rules of the game $G^{\psi_x}_{\jty,\exists}(g^{\uparrow\downarrow})$, we have $g^{\uparrow\downarrow} \leq \psi_x(c)$. Therefore, $V', g \Vdash \psi_x$. By the induction hypothesis, at $(g, \psi_x)$ there exists a strategy $\sigma_g$ which is winning for $\exists$. Since $g$ was arbitrary such that $g^{\uparrow\downarrow} \leq c$, the strategy $\bigcup_{g\in\val{c}}\sigma_{g}$ is winning for $\exists$ from position $(g',\varphi)$ for any $g'\in\val{c}$. Indeed, either from one point onwards $\exists$ follows a specific strategy $\sigma_g$, which is winning for her, or $x$ repeats in the play infinitely often, which means that $\exists$ wins.
The proof for $m\succ \varphi$ is similar to e.g.\ $m\succ \Box \psi$ from Theorem \ref{lem:non_distributive_modal_logic_game}.

($\Leftarrow$) Assume that  $(g, \phi)$ is winning for $\exists$ and let $\sigma$ be a winning strategy. We show that $g \in \val{c}$ by showing that the game $G^{\psi_x}_{\jty,\exists}(g)$ is winning for $\exists$.

Let $\mathcal{T}$ be the game tree generated by the game starting from $(g,\varphi)$ in which $\exists$ strictly follows $\sigma$. Let 

{{\centering
  $S_{\mathcal{T}}=\{g'\in G \mid (g',x)\in\mathcal{T}\  \&\  ((g'',\psi)<_\mathcal{T}(g',x)\Rightarrow \psi\neq x)\}$
  \par
}}
\smallskip
be the set of the elements of $G$, such that a position $(g',x)$ with the second component $x$ appearing for the first time in a play according to $\sigma$ from $(g, \phi)$.

Let $P'=P\cup\{x\}$ and notice that if we let $V'(x)=S^{\uparrow\downarrow}_\mathcal{T}$, then $\sigma$ is winning for $(g,\psi_x)$, and by induction hypothesis $V', g\Vdash \psi_x$. Hence $g^ {\uparrow\downarrow}\leq \psi_x(\bigvee S_{\mathcal{T}})$. Then the set $S_{\mathcal{T}}$ is a legitimate move for the game $G^{\psi_x}_{\jty,\exists}(g)$. Now for any choice at $g'\in S_{\mathcal{T}}$ that $\forall$ makes, we can proceed recursively using the sub-tree of $\mathcal{T}$ generated by $(g,\varphi)$.

This process recursively defines a strategy $\sigma'$ for $\exists$ for the game $G^{\psi_x}_{\jty,\exists}(g)$, which is winning since it provides a play at every step and if the game continues infinitely $\exists$ wins. But then $g$ is a winning position for $\exists$ and by Lemma \ref{lem:unfolding_join_exists} $g\in\val{c}$, i.e. $g\Vdash \varphi$.

The case for $\mu x . \psi_x$ is verbatim the same, defining a winning strategy for $\forall$ using the game $G^f_{\mty,\forall}(\cnomm)$. 
This concludes the proof.

\paragraph{\bf Proof of Theorem \ref{thm:simulation invariance}.}
We use game semantics of non-distributive modal $\mu$-calculus for the proof. We show that if $\exists$ has a winning strategy in a position $(g_1, \phi)$ (resp.~$(m_1,\phi)$) in the evaluation game for the model $\mathbb{M}_1$, then  $\exists$ has a winning strategy in a position $(g_2, \phi)$ (resp.~$(m_2,\phi)$) in the evaluation game for the model $\mathbb{M}_2$.
The proof is by induction on the length of the formula $\phi$.

\textbf{Base case:} The proof for the cases when formula $\phi$ is a proposition follows immediately from the items 1 and 2 of Definition \ref{def:simulation}. Suppose $\phi$ is of the form $\top$. The proof for 1 is trivial. Suppose $\exists$ has a winning strategy in $(m_1, \top)$. Then, there exists $g_1 \in G_1$ such that $g_1 {I_1}^c m_1$. By item 4 of Definition \ref{def:simulation}, there exists $g_2$ such that $g_2 {I_2}^c m_2$. Therefore, $\mathbb{M}_2, m_2 \not\succ \top$ which is is equivalent to $\exists$  having a winning strategy in position $(m_2, \top)$ in the evaluation game for $\mathbb{M}_2$.  The proof for $ \phi = \bot$ is similar.

\textbf{Induction case:} We consider different cases with a different outermost connectives. 

\noindent (a) Suppose $\phi=\phi_1 \vee \phi_2$ for some $\phi_1$, $\phi_2$. Suppose  $\exists$ has a winning  strategy  at $(m_1, \phi_1 \vee \phi_2)$. Then  $\exists$ has a winning  strategy  at $(m_1, \phi_1)$ or $(m_1, \phi_2)$.
As $m_1 T m_2$, by induction hypothesis, $\exists$ has a winning  strategy  at  $(m_2, \phi_1)$ or $(m_2, \phi_2)$ which is equivalent to $\exists$ having a winning  strategy in $(m_2, \phi_1 \vee \phi_2)$.    

Suppose  $\exists$ has a winning  strategy  at $(g_1, \phi_1 \vee \phi_2)$, and   $\exists$ does not have a winning  strategy in position $(g_2,\phi_1 \vee \phi_2)$. Then by the evaluation game, there exists $m_2$ such that $g_2 I_2^c m_2$ and $\exists$ does not have  winning strategy in position $(m_2,\phi_1 \vee \phi_2)$. By  item 3 of Definition \ref{def:simulation},  there exists $m_1 \in M_1$ such that  $g_1 I_1^c m_1$, and $m_1 T m_2$. By the contra-position of the previous proof,  $\exists$ does not have  winning strategy in position $(m_2,\phi_1 \vee \phi_2)$ implies $\exists$ does not have  winning strategy in position  $(m_1,\phi_1 \vee \phi_2)$. Therefore, $\exists$ does not have a winning  strategy  at $(g_1, \phi_1 \vee \phi_2)$. This is a contradiction.

\noindent (b) Proof for  $\phi=\phi_1 \wedge \phi_2$ is dual to the proofs in case (a). 

\noindent (c) Suppose $\phi=\Diamond \phi_1  $ for some $\phi_1$. Suppose $\exists$ has a winning  strategy  in $(m_1,\Diamond \phi_1)$. Then, $\exists$ has a winning  strategy  in $(g_1, \phi_1)$  for some $m_1 {R_\Diamond}_1^c g_1$. Then, by item 6 of Definition \ref{def:simulation},  there exists $g_2$ such that $m_2 {R_\Diamond}_2^c g_2$ and $g_1 S g_2$. By induction hypothesis, $\exists$ has a winning strategy in $(g_2, \phi_1)$. Therefore, $\exists$ has a winning  strategy  in $(m_2, \Diamond \phi_1 )$.

Suppose  $\exists$ has a winning  strategy  at $(g_1, \Diamond \phi_1)$, and   does not have a winning  strategy in position $(g_2, \Diamond \phi_1)$. Then by the evaluation game, there exists $m_2$ such that $g_2 {I}_2^c m_2$ and $\exists$ does not have  winning strategy in position $(m_2,\Diamond \phi_1)$. By  item 3 of Definition \ref{def:simulation},  there exists $m_1 \in M_1$ such that  $g_1 I_1^c m_1$, and $m_1 T m_2$. By the contra-position  of previous proof,  $\exists$ does not have  winning strategy in position $(m_2,\Diamond \phi_1)$ implies $\exists$ does not have  winning strategy in position  $(m_1,\Diamond \phi_1)$. Therefore, $\exists$ does not have a winning  strategy  at $(g_1,\Diamond \phi_1)$. This is a contradiction. 

\noindent (d) Proof for  $\phi=\Box \phi_1$ is dual to the proofs in case (c).

\noindent (e) Suppose $\phi=  \nu x.\phi_x$ for some $\phi_x$. Suppose $\exists$ has a winning  strategy  in $(g_1,\nu x.\phi_x)$.  Then, $\exists$ has a winning  strategy  in $(g_1, \phi_x)$. We consider two sub-cases. (i) Suppose the winning strategy of $\exists$ does not contain infinite repetitions of $x$, then w.l.o.g.~we can assume that the winning strategy contains no repetitions of $x$. In this case, the  play starting from $(g_1, \phi_x)$ in the winning strategy is identical to the play in the case $x$ is treated as a propositional variable (i.e.~as a free variable). Thus, by induction we can conclude that  $\exists$ has a winning  strategy  in $(g_2, \phi_x)$. Therefore, $\exists$ has a winning  strategy  in $(g_2,\nu x.\phi_x)$.   (ii)  Suppose the winning strategy of $\exists$  contains infinite repetitions of $x$. We use $k_i$ to denote an element belonging to either $G_i$ or $M_i$ and $U$ to denote the relation $S$ or $T$. From the previous proofs we know the followings: 
(1) Suppose $(m_1, \psi)$ and  $(m_2, \psi)$ are any two nodes with $m_1 T m_2$ and $\psi \neq x$. Then for every possible move $(m_1, \psi) \to (k_1, \psi')$ which $\exists$ can make, there exists a possible  move  $(m_2, \psi) \to (k_2, \psi')$ for $\exists$,  such that $k_1 U k_2$. (2) Suppose $(g_1, \psi)$ and  $(g_2, \psi)$ are any two nodes with $g_1 S g_2$ and $\psi \neq x$. Then for every possible move $(g_2, \psi) \to (k_2, \psi')$ which $\forall$ can make, there exists a possible  move  $(g_1, \psi) \to (k_1, \psi')$ for $\forall$,  such that, $k_1 U k_2$. Therefore, for any play from $(g_1,\nu x.\phi_x)$ to $(g_1',x)$ in which $\exists$ follows some winning strategy, $\exists$ has a strategy to take the play from $(g_2,\nu x.\phi_x)$ to $(g_2',x)$  for  some $g_1 S g_2$.  Then after two automatic (without player assigned) moves in both games we reach the points $(g_1', \phi_x)$  and $(g_2', \phi_x)$, respectively, where $g_1' S g_2'$. In this way, $\exists$ can use the strategy which forces the play in the game for model $\mathbb{M}_1$ to be infinite to get a strategy which forces the  play in the game for model $\mathbb{M}_2$ to be infinite and win.

Suppose  $\exists$ has a winning  strategy  at $(m_1, \nu x.\phi_x)$. Then by the evaluation game, there exists $g_1$ such that $g_1 {I}_1^c m_1$ and $\exists$ has a   winning strategy in position $(g_2, \nu x.\phi_x)$. By  item 4 of Definition \ref{def:simulation},  there exists $g_2 \in G_1$ such that  $g_2 I_2^c m_2$, and $g_1 S g_2$. By the  proof of the  previous  part,  $\exists$  has  winning strategy in position $(g_2,\nu x.\phi_x)$. 

\noindent (f) Proof for  $\phi= \nu x.\phi_x$ is dual to the proofs in case (e). This concludes the proof.

\section{Unfolding games with meet generators}\label{app:B}
Dual to the unfolding games using join-generators, we can also define unfolding games with meet-generators of lattices with   the following game board:
\vspace{-0.2 cm}
\begin{center}
    \begin{tabular}{|c|c|c|}
    \hline
      \textbf{state}& \textbf{player} & \textbf{admissible
         moves} \\
         \hline
         $\cnomm\in \mty(\mathbb{L})$  & $\forall$ & $\{S \subseteq \mty(\mathbb{L}) \mid  f(\bigwedge S) \leq \cnomm\}$\\
         $S \subseteq \mty(\mathbb{L})$ & $\exists$ & $S$\\
         \hline
     \end{tabular}
\end{center}
\vspace{-0.3 cm}

We denote the game  in which infinite games are won by $\forall$ (resp.~$\exists$)   by  $G^f_{\mty,\forall}(\cnomm)$ (resp.~ $G^f_{\mty,\exists}(\cnomm)$), where  $\cnomm \in \mty(\mathbb{L})$ is the starting position of the game. Items 1 and 2 of the  following lemma are proved in a manner analogous to the Lemmas \ref{lem:unfolding_join_exists}, and \ref{lem:unfolding_join_forall}, respectively.

\begin{lemma}\label{lem:unfolding_meet_exists}
\begin{enumerate}[noitemsep, nolistsep]
    \item Let $W=\{\cnomm \in \mty(\mathbb{L}) \mid\forall\mbox{ has a winning strategy for } G^f_{\mty,\forall}(\cnomm)\}$. Then $\bigwedge W$ is the least fixed point of $f$. 
    \item Let $W=\{ \cnomm \in \mty(\mathbb{L}) \mid\forall \mbox{ has a winning strategy for } G^f_{\mty,\exists}(\cnomm)\}$. Then $\bigwedge W$ is the greatest  fixed point of $f$. 
\end{enumerate}
\end{lemma}

\section{Considerations on generalizing unfolding games}\label{appendix:C}
Shifting from a power-set algebra to a complete lattice offers a number of possible generalizations. Firstly, in a power-set algebra the atoms join-generate the Boolean algebra. Secondly, every subset of atoms correspond to a unique element of the Boolean algebra. In arbitrary complete lattice we might not have enough join irreducibles (and certainly not enough join primes or atoms) to generate the lattice and the joins of different sets of join generators may correspond to the same element. Hence, one might consider different generalizations of the game, such as game $G'$ presented in the following table:

\vspace{-0.2 cm}
\begin{center}
    \begin{tabular}{|c|c|c|}
    \hline
      \textbf{state}& \textbf{player} & \textbf{admissible
         moves} \\
         \hline
         $\nomj\in \jty(\mathbb{L})$  & $\exists$ & $\{u\in\mathbb{L} \mid \nomj \leq f(u)\}$\\
         $u\in\mathbb{L}$ & $\forall$ & $\{\nomj\in\jty(\mathbb{L}) \mid\nomj\leq u\}$\\
         \hline
     \end{tabular}
\end{center}
\vspace{-0.2 cm}

\begin{wrapfigure}{r}{1cm}
   \begin{tikzcd}[sep=tiny]
	& \top \\
	\vdots \\
	{a_n} \\
	\vdots && b \\
	{a_2} \\
	{a_1} \\
	& \bot
	\arrow[no head, from=7-2, to=6-1]
	\arrow[no head, from=6-1, to=5-1]
	\arrow[no head, from=4-1, to=3-1]
	\arrow[no head, from=3-1, to=2-1]
	\arrow[no head, from=7-2, to=4-3]
	\arrow[no head, from=4-3, to=1-2]
	\arrow[no head, from=2-1, to=1-2]
\end{tikzcd}
\end{wrapfigure}

This game may seem more natural. Player $\exists$ instead of choosing a set of join generators picks an element of the lattice. The two game boards actually coincide in the case of power-set algebras where the set of join generators is the set of atoms. However, if $\mathbb{L}$ is infinite and non-distributive, then one cannot prove an analogue of Lemma \ref{lem:unfolding_join_exists}. Consider the adjoining  lattice with a map $f$ such that  $f(\bot)=a_1$, $f(a_n)=a_{n+1}$, and $f(b)=f(\top)=\top$. Then, from  $b$, $\exists$ has only two possible moves ($\top$ or $b$) and from both of them $\forall$ can play $b$. Hence, $\forall$ has a winning strategy for $b$, which is below the least fixed point of $f$.

On the other hand, if $\mathbb{L}$ is finite or distributive and $\jty(\mathbb{L})$ contains only join prime elements, then we can prove a version of Lemma \ref{lem:unfolding_join_exists} as follows:

\begin{lemma}\label{lem:unfolding_join_forallalternative}
Let $W=\{ \nomj \in \jty(\mathbb{L}) \mid \exists \mbox{ has a winning strategy for }G'_{\jty,\forall}(\nomj)\}$. Then $\bigvee W$ is the least fixed point of $f$. 
\end{lemma}
\begin{proof}
    For every ordinal $\alpha$, we will show that $\{\nomj\in\jty(\mathbb{L})\ \mid\ \nomj\leq f^{\alpha}(\bot)\}\subseteq W$. We do this by recursion on $\alpha$. The statement is vacuously true for $\alpha=0$. Assume that $\{\nomj\in\jty(\mathbb{L})\ \mid\ \nomj\leq f^{\alpha}(\bot)\}\subseteq W$ is the case, then for any $\nomj\leq f(f^{\alpha}(\bot))$, playing $f^{\alpha}(\bot)$ is a winning move for $\exists$, hence $\{\nomj\in\jty(\mathbb{L})\ \mid\ \nomj\leq f^{\alpha+1}(\bot)\}\subseteq W$. In case $\mathbb{L}$ is finite, we know that the sequence $f^{\alpha}(\bot)$ is finite, hence we are done. Now let $\alpha$ be a limit ordinal and assume that $\{\nomj\in\jty(\mathbb{L})\ \mid\ \nomj\leq f^{\beta}(\bot)\}\subseteq W$ for every $\beta<\alpha$. Then $f^{\alpha}(\bot)=\bigvee f^{\beta}(\bot)$. It follows that if $\nomj\leq f^{\alpha}(\bot)$, since $\nomj$ is a join prime element, there exists a $\beta<\alpha$ such that $\nomj\leq f^{\beta}(\bot)$, i.e.\ $\nomj\in W$. This concludes the proof.
\end{proof}

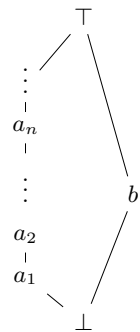
\end{document}